\documentclass[letterpaper,12pt]{article}                
\usepackage[utf8]{inputenc}
\usepackage{amssymb}
\usepackage{amsmath}
\usepackage[cmtip,all]{xy}
\usepackage[english]{babel}
\usepackage{tikz}
\usepackage{amsthm}
\newtheorem*{thm}{Theorem}
 \newtheorem*{fl}{Fundamental lemma}
\newtheorem{lemme}{Lemma}

\newtheorem{prop}{Proposition}
\newtheorem{cor}{Corollary}
\newtheorem{ex}{Example}
\newtheorem{cex}{Counter-example}
\newtheorem*{q}{Question}

\usepackage{hyperref}

\begin{document}
\title{Alternate modules are subsymplectic}

\author{Cl\'ement Gu\'erin\thanks{
              IRMA Institut de Recherche  Math\'ematiques avanc\'ee,  University of Strasbourg, 7 rue Ren\'e Descartes, F-67084 Strasbourg , France,
            e-mail : guerin@math.unistra.fr} }
\maketitle

\begin{abstract}
In this paper, an alternate module $(A,\phi)$ is a finite abelian group $A$ with a $\mathbb{Z}$-bilinear application $\phi:A\times A\rightarrow \mathbb{Q}/\mathbb{Z}$ which is alternate (i.e. zero on the diagonal). We shall prove that any alternate module is subsymplectic, i.e. if $(A,\phi)$ has a Lagrangian of cardinal $n$ then there exists an abelian group $B$ of order $n$ such that $(A,\phi)$ is a submodule of the standard symplectic module  $B\times B^*$.

\end{abstract}

\textbf{Keywords : } alternate modules, symplectic modules, Lagrangians, finite abelian groups.

\textbf{Mathematics Subject Classification : }20K01.
\tableofcontents

\section{Definitions and statement of the result}\label{3.sec1}

\bigskip

Before stating the result, we give some standard definitions along with some properties regarding the alternate modules. The concept is defined by Wall in \cite{Wal} and by Tignol, Amitsur in \cite{A-T2}. Among the results, we will establish here, some might be directly read in those references and also in \cite{T-W}.

\bigskip

Let   $A$ be an abelian group. If we are given an application   $\phi:A\times A\rightarrow \mathbb{Q}/\mathbb{Z}$, then we say that $(A,\phi)$ is a  \textbf{bilinear module} if $\phi$ is bilinear when $A$ is seen as a $\mathbb{Z}$-module. By definition, the \textbf{underlying group} of the bilinear module   $(A,\phi)$ is  $A$ and  the associated \textbf{bilinear form} to $(A,\phi)$ is  $\phi$. 

\bigskip

In the sequel, all bilinear modules are assumed to be finite.

\bigskip

Let $(A,\phi)$ a  bilinear module, we say that $(A,\phi)$ is an \textbf{alternate module} if $\phi(a,a)=0$ for all  $a\in A$. If $(A,\phi)$  is an alternate module, then for all $a,b\in A$ :

\begin{align*}
0&=\phi(a+b,a+b)&\text{ since $(A,\phi)$ is alternate}\\
&=\phi(a,a)+\phi(a,b)+\phi(b,a)+\phi(b,b)&\text{ by  "bilinearity"}\\
&=\phi(a,b)+\phi(b,a)&\text{ since $(A,\phi)$ is alternate}
\end{align*}

Whence for all  $a,b\in A$,   $\phi(a,b)=-\phi(b,a)$. We will sum up this by saying that if $(A,\phi)$ is an alternate module then, in particular $\phi$ is anti-symmetric.

\bigskip

Let $(A,\phi)$ be an alternate module and $a,b\in A$,  we say that  $a$ is  \textbf{orthogonal} to $b$ if $\phi(a,b)=0$. Since  $A$ is alternate, any element of  $A$ is orthogonal to itself and the relation of orthogonality is symmetric by anti-symmetry of the bilinear form  $\phi$.

\bigskip

Let $(A,\phi)$ be an alternate module and $S$ a subset of  $A$. We define $S^{\perp}$ the \textbf{orthogonal} of $S$ as the subset of elements in  $A$ which are orthogonal to  $S$. One can directly check that $S^{\perp}$  is a subgroup of $A$.

\bigskip

Let $(A,\phi)$ be an alternate module and $B$ a subgroup of  $A$ then the  bilinear module induced by $A$ on $B$ : $(B,\phi_{B\times B})$ is also an alternate module  called the  \textbf{induced submodule} on  $B$ by  $\phi$.

\bigskip

Let $(A_1,\phi_1)$ and $(A_2,\phi_2)$  be two  alternate modules, we define the  \textbf{orthogonal sum} of those modules by :

$$(A_1,\phi_1)\overset{\perp}{\oplus}(A_2,\phi_2):=(A,\phi) $$ 

With  $A:=A_1\oplus A_2$ and $\phi((a_1,a_2),(b_1,b_2)):=\phi_1(a_1,b_1)+\phi_2(a_2,b_2)$. It is a straightforward verification that  $(A,\phi)$ is an alternate module, verifying  $A_1\leq A_2^{\perp}$, $A_2\leq A_1^{\perp}$ and  the induced  submodule by  $\phi$ on  $A_i$ is $(A_i,\phi_i)$. 

\bigskip

Conversly, if  $(A,\phi)$ is an alternate module and $B,C$ two subgroups of  $A$ such that $A=B\oplus C$ with  $B\leq C^{\perp}$, then  $(A,\phi)$ can be written as the orthogonal sum of the two induced modules on  $B$ and $C$ by $\phi$. 

\bigskip

\bigskip

Let  $A$  be a finite  module, we define $A^*$, the \textbf{dual} of  $A$ as the group $A^*:=Mor(A,\mathbb{Q}/\mathbb{Z})$. A non-trivial result (direct consequence of the classification of finite abelian groups) states that $A$ is isomorphic to  $A^*$ as a group. Remark that the isomorphism between $A$ and $A^*$ is not canonical.

\bigskip

Let $(A,\phi)$ be an alternate module, we define the \textbf{dual application} associated to $(A,\phi)$ as  :

\begin{displaymath}
\phi^*:
\left|
  \begin{array}{rcl}
    A& \longrightarrow &A^*\\
 a& \longmapsto &b\mapsto \phi(a,b)\\
  \end{array}
\right.
\end{displaymath}

\bigskip

The  \textbf{kernel} $K_{\phi}$ of $(A,\phi)$ is, by definition, the kernel  $Ker(\phi^*)$ of the dual application $\phi^*$. In other words,  $K_{\phi}=A^{\perp}$ i.e. the set of elements in  $A$ which are orthogonal to any element of  $A$.   

\bigskip

Let $(A,\phi)$ be an alternate module, we say that  $(A,\phi)$ is a  \textbf{symplectic module} if its kernel $K_{\phi}$ is trivial (i.e. the application $\phi$ is non-degenerate).

\bigskip

To any   alternate module $(A,\phi)$, we can associate a  symplectic module  $(B,\phi_B)$ setting $B:=A/K_{\phi}$ and :

$$\phi_B(a_1 \text{ mod } K_{\phi},a_2 \text{ mod } K_{\phi}):=\phi(a_1,a_2)$$

The fact that $(B,\phi_B)$ is well defined and is a   symplectic module is clear. This symplectic module will always be denoted $(A/K_{\phi},\overline{\phi})$ and called the \textbf{symplectic module} associated to  $(A,\phi)$. 

\bigskip

An important example of symplectic module is the following :

\begin{ex}\label{3.sympmod}

Let $B$ be an abelian group. We define $A:=B\times B^*$ and :

\begin{displaymath}
\phi:
\left|
  \begin{array}{rcl}
    A\times A& \longrightarrow &\mathbb{Q}/\mathbb{Z}\\
 ((a_1,\psi_1),(a_2,\psi_2))& \longmapsto & \psi_2(a_1)-\psi_1(a_2)\\
  \end{array}
\right.
\end{displaymath}

Then  $(A,\phi)$ is a  symplectic module. In the sequel, such  symplectic module will always be denoted $B\times B^*$ (the underlying bilinear form being the form $\phi$ as above).

\end{ex}

\begin{proof}
Clearly $(A,\phi)$ is an alternate module. Let $(a,\psi)$ be in $K_{\phi}$ then for all  $b\in A$   :

$$\phi((a,\psi),(b,0))=0\Rightarrow \psi(b)=0 $$

This implies that  $\psi$ is the trivial morphism, i.e. $\psi=0$. Furthermore for all $\psi'\in Mor(A,\mathbb{Q}/\mathbb{Z})$ :

$$\phi((a,0),(0,\psi'))=0\Rightarrow \psi'(a)=0$$

Since  $A$ is isomorphic to $Mor(A,\mathbb{Q}/\mathbb{Z})$ (although the isomorphism is not canonical) this implies that $a=0$. Therefore $K_{\phi}$ is trivial and $(A,\phi)$ is a  symplectic module.\end{proof}

Let $(A,\phi)$  be an alternate module and $S$ a subset of $A$. We say that   $S$ is \textbf{isotropic} if $S\subseteq S^{\perp}$. Furthermore, if   $S=S^{\perp}$ then  $S$ is a subgroup of $A$ and is called a   \textbf{Lagrangian} of $(A,\phi)$. In general (see proposition \ref{3.carlag}) the cardinal of a Lagrangian of $(A,\phi)$ only depends on $|A|$ and $|K_{\phi}|$. We shall denote $n_{A,\phi}:=\sqrt{|A||K_{\phi}|}$ the cardinal of any Lagrangian of $(A,\phi)$. 

\bigskip

In the example \ref{3.sympmod}, the subgroups  $B$ and  $B^*$ of $A$ are both Lagrangians of the module  $(A,\phi)$.

\bigskip

Let $(A,\phi)$   and  $(A',\phi')$ be alternate modules, we say that  $(A,\phi)$ and $(A',\phi')$ are  \textbf{isometric} if there exists a group isomorphism   $g:A\rightarrow A'$ verifying $\phi=g^*\phi'$.  The isometric relation is clearly an equivalence relation.

\bigskip
 
Now, we come to the most important definition of this paper. Let $(A,\phi)$ be an alternate module. We say that $(A,\phi)$ is \textbf{subsymplectic} if   there exists an abelian group $B$ of order $n_{A,\phi}$ such that $(A,\phi)$ is included in the standard symplectic module $B\times B^*$. The result we are going to prove here is   :

\begin{thm}\label{result}
Any alternate module is subsymplectic.

\end{thm}

Whereas the classification of symplectic modules is easily delt with, it seems hopeless to do the same thing for alternate modules in general. The statement that alternate modules are subsymplectic appears as the closest possible result to a classification of alternate modules.

\bigskip

In another paper, we will use this result to classify conjugacy classes of centralizers of irreducible subgroups in $PSL(n,\mathbb{C})$. 

\bigskip

In the second section, we shall characterize and study the Lagrangians in alternate modules.  In the third section we will prove the theorem. In the last section, we make some remarks about this proof and alternate modules in general.

\bigskip

\section{Lagrangians in alternate modules}

\bigskip

Some results in this section have been proven in other papers, we shall only give the references. We begin with an elementary propoosition :

\begin{prop}\label{3.sympB}

Let   $(A,\phi)$ be a  symplectic module and $B$ a subgroup of  $A$ then :

\begin{enumerate}

\item $|B||B^{\perp}|=|A|$ and $(B^{\perp})^{\perp}=B$.

\item If the induced submodule on $B$ by $(A,\phi)$ is a  symplectic module then the induced submodule on $B^{\perp}$ by  $(A,\phi)$ is also a symplectic module, Furthermore 

$$A=B\overset{\perp}{\oplus} B^{\perp}$$

\end{enumerate}

\end{prop}

\begin{proof}
See proposition 2.2 of \cite{A-T2} or  lemma 1 in \cite{Wal}.\end{proof}

Whereas, alternate modules are very hard to classify, the following   property implies that any symplectic module is isometric to one constructed as in the example \ref{3.sympmod}. In particular, the set of symplectic modules is quite rigid.

\begin{cor}\label{3.classificationsymp}

Let  $(A,\phi)$ be a symplectic module. There exists $B\leq A$ such that $(A,\phi)$ is isomorphic to $B\times B^*$.

\end{cor}

\begin{proof}
See lemma 7 in \cite{Wal} or theorem 4.1 in \cite{A-T2}. \end{proof}
 
This construction can be seen as the analogous of a symplectic base for $(A,\phi)$ (i.e. a "base" of $B$ as a finite abelian group and its "dual base" in $B^*$). A direct consequence of this corollary is that any two symplectic modules $(A,\phi)$ and $(A',\phi')$ are isometric if and only if $A$ and $A'$ are isomorphic as groups.

\bigskip

It is, somehow, surprising that such  symplectic base exists for those finite modules (in a similar fashion than for bilinear forms on $K$-vector spaces). Whereas $K$-vector spaces endowed with an alternate form (not necessarily symplectic) are always the sum of the kernel of the form and a non-degenerate complementary, this is not the case for alternate modules in general. Indeed, one can consider the following counter-example :

\begin{cex}\label{3.pasdebasesymp}

Let $A:=\mathbb{Z}/2\times\mathbb{Z}/4\times \mathbb{Z}/8$. We denote $e_1:=(1,0,0)$, $e_2:=(0,1,0)$ and $e_3:=(0,0,1)$. Then, on  the base  $e_1,e_2,e_3$, we define the alternate bilinear form $\phi$ on  $A$ with its associated matrix :

$$M_{\phi}:=\begin{pmatrix}0&1/2&1/2\\1/2&0&-1/4\\1/2&1/4&0\end{pmatrix}  $$

Then the kernel $K_{\phi}$ of $(A,\phi)$ is not a direct factor of  $A$, in particular, $A$ is not the sum of a non-degenerate module and its kernel.

\end{cex}

\begin{proof}

With a straightforward computation, one can check that $K_{\phi}=\langle e_1+2e_2+2e_3\rangle$ is isomorphic to $\mathbb{Z}/4$. We remark that $2(e_1+2e_2+2e_3)=4e_3\in K_{\phi}$. Furthermore if $(x,y,z)\in A$ then $4(x,y,z)=(0,0,4z)$ so any element $a\in A$ of order $8$ will verify that $4a=4e_3\in K_{\phi}$. In particular, $K_{\phi}$ is not a direct factor of $A$.  \end{proof}

This little example shows that the kernel of  $(A,\phi)$ is not necessarily a direct factor of $A$. Furthermore, one can check that $A$ cannot be written as the orthogonal sum of strictly smaller submodules. In some sense, it is irreducible, this suggests that if there is a classification of alternate modules (in some sense, we would like to have one in the sequel), it should be complicated. Another corollary of the proposition \ref{3.sympB} :

\begin{cor}\label{3.Lsm}

Let $(A,\phi)$ be a symplectic module, then any Lagrangian $L$ of $A$ is of cardinal  $\sqrt{|A|}$.

\end{cor}

\begin{proof}

By definition,  $L=L^{\perp}$, applying the first point of proposition \ref{3.sympB}, we get $|L|^2=|A|$. \end{proof}

In their paper, \cite{A-T2}, J.-P. Tignol and  S.A. Amitsur  are interested in Lagrangian of symplectic modules, they apply this to division algebra (cf. \cite{A-T1}). We need to understand, Lagrangians of alternate modules. We recall some results about Lagrangians in this case :

\begin{prop}\label{3.Lagrsymp}
Let    $(A,\phi)$  be an alternate module. The Lagrangians of  $A$ are exactly the maximal elements for the inclusion within the set of  isotropic subsets of $A$. In other words a subset  $L$ in $A$ is a Lagrangian of  $A$ if and only if for any  $K'$ isotropic subset of $A$ containing $L$,  $K'=L$.

\end{prop}

\begin{proof}

Let  $L$ be an isotropic subset of $A$. Assume that $L$ is a Lagrangian, then if  $K'$ is an isotropic subset of $A$ containing $L$, then   $L^{\perp}\supseteq K'$ since $K'$ is orthogonal to $K'$ whence to $L$. Since  $L$ is a Lagrangian, $L=L^{\perp}$ whence $K'\subseteq  L$, since $K'\supseteq L$ by assumption $L=K'$ and $L$ is thus maximal among the isotropic subsets of  $A$.

\bigskip

Assume that $L$ is not a Lagrangian. Since  $L$ is isotropic,   $L\leq L^{\perp}$ and there exists  $a\in L^{\perp}$ which does not belong to $L$. We define $K':=\langle L,a\rangle$, clearly  $L\leq K'$ and $K'\neq L$, furthermore $L$ being isotropic is orthogonal to  $a$ by assumption, whence $K'$ is an isotropic subspace of $A$. Whence $L$ is not maximal among the isotropic subsets of $A$. \end{proof}

Let $(A,\phi)$ be an alternate module then the set of isotropic subspaces of $A$ is not empty (it contains the trivial group) and finite (since $A$ is finite). Therefore, there exists a maximal element for the inclusion. Whence any alternate module admits a Lagrangian. Furthermore, if $K$ is any isotropic subset, there will always be a maximal isotropic subset containing it, therefore any isotropic subset is contained in a Lagrangian. In particular, for any $a\in A$, there exists a Lagrangian containing $a$ (applying what we have just done to $K=\langle a\rangle$).

\bigskip

The next proposition is a generalization of corollary  \ref{3.Lsm} :

\begin{prop}\label{3.carlag}
Let  $(A,\phi)$ be an alternate module. For any Lagrangian  $L$  of $A$, the cardinal of  $L$ is $\sqrt{|A||K_{\phi}|}$.

\end{prop}

\begin{proof}

We define $(A/K_{\phi},\overline{\phi})$ to be the associated symplectic module to $(A,\phi)$ and $\pi$ be the projection of $A$ onto $A/K_{\phi}$. Clearly the application $\pi$ leads to a bijective correspondance between isotropic subgroups of $A$ containing $K_{\phi}$ and isotropic subgroups of $A/K_{\phi}$. 

\bigskip

Since maximal elements among isotropic subsets of $A$ are isotropic subgroups of $A$ containing $K_{\phi}$, it follows that $\pi$ induces a bijective correspondance between Lagrangians in $A$ and Lagrangians in $A/K_{\phi}$ (by sending $L$ to $\pi(L)$). In particular, any Lagrangian $L$ of $A$ is of the form $\pi^{-1}(\overline{L})$ where $\overline{L}$ is a Lagrangian in $A/K_{\phi}$.

\begin{align*}
|L|&=|\pi^{-1}(\overline{L})|\\
&=|K_{\phi}||\overline{L}|\\
&=|K_{\phi}|\sqrt{\frac{|A|}{|K_{\phi}|}}\text{ by corollary  \ref{3.Lsm} }\\
&=\sqrt{|A||K_{\phi}|}
\end{align*}

\end{proof}

We see that the cardinal of Lagrangians is constant in alternate modules. However, the isomorphism class of Lagrangians may vary (see also \cite{A-T2}).
 
\begin{ex}\label{3.lagrdiff}

Let $(A,\phi)$ be the alternate module defined in counter-example \ref{3.pasdebasesymp}. Then $A=\mathbb{Z}/2\times \mathbb{Z}/4\times \mathbb{Z}/8$, and if we denote $e_1:=(1,0,0)$, $e_2=(0,1,0)$ and $e_3=(0,0,1)$ in $A$ then the subgroups $L_1:=\langle e_3,e_1+2e_2\rangle$ and $L_2:=\langle e_1,2e_2,2e_3\rangle$ are Lagrangians in $A$ and  :

\begin{align*}
L_1\text{ is isomorphic to }& \mathbb{Z}/2\times\mathbb{Z}/8\\
L_2\text{ is isomorphic to }& \mathbb{Z}/2\times\mathbb{Z}/2\times\mathbb{Z}/4
\end{align*}

In particular, $L_1$ is not isomorphic to  $L_2$.

\end{ex}

\begin{proof}
From the counter-example \ref{3.pasdebasesymp}, $|K_{\phi}|=4$, hence the cardinal of a Lagrangian in  $(A,\phi)$ is $\sqrt{2\cdot4\cdot8\cdot4}=16$ using proposition  \ref{3.carlag}.

\bigskip

We directly check that both $L_1$ and $L_2$ are isotropic.

\bigskip

Furthermore $e_1+2e_2$ is of order $2$ and not in $\langle e_3\rangle$, therefore $L_1$ is isomorphic to $\langle e_1+2e_2\rangle\times\langle e_3\rangle$. Whence $L_1$ is isomorphic to $\mathbb{Z}/2\times\mathbb{Z}/8$. In particular, $L_1$ is isotropic and has the cardinal of a Lagrangian, by maximality of Lagrangians (proposition \ref{3.Lagrsymp}) it is Lagrangian.

\bigskip

Finally,   $L_2$ is clearly isomorphic to $ \mathbb{Z}/2\times\mathbb{Z}/2\times\mathbb{Z}/4$. Likewise, the  $L_2$ is isotropic of cardinal $16$,  by maximality of Lagrangians (proposition \ref{3.Lagrsymp}) it is Lagrangian.\end{proof}
 
In general, it is much more convenient to work with abelian $p$-groups than finite abelian groups. We remark that :

\begin{prop}\label{decomporth}

Let $(A,\phi)$ be an alternate module such that $(A,\phi)$ is isometric to the orthogonal sum of  $(A_1,\phi_1)$ and $(A_2,\phi_2)$ :

$$(A,\phi)=(A_1,\phi_1)\overset{\perp}{\oplus}(A_2,\phi_2) $$

If both $(A_1,\phi_1)$ and $(A_2,\phi_2)$ are subsymplectic then  $(A,\phi)$ is subsymplectic.

\end{prop}

\begin{proof}

We begin by showing that $K_{\phi}=K_{\phi_1}\oplus K_{\phi_2}$. Clearly, $K_{\phi}$ contains $K_{\phi_1}$ and $K_{\phi_2}$ and both are in direct sum. Finalely, if $k=(l_1,l_2)\in K_{\phi}\leq A_1\oplus A_2$ then for all  $a_1\in A_1$, we have :

$$0=\phi(k,(a_1,0))=\phi_1(l_1,a_1)+\phi_2(l_2,0)=\phi_1(l_1,a_1) $$

This is true for any $a_1\in A_1$ so  $l_1\in K_{\phi_1}$. Likewise, $l_2\in K_{\phi_2}$, whence  $K_{\phi}=K_{\phi_1}\oplus K_{\phi_2}$.

\bigskip

It follows that $|K_{\phi}|=|K_{\phi_1}||K_{\phi_2}|$, since $|A|=|A_1||A_2|$, we have :

\begin{equation}
n_{A,\phi}=n_{A_1,\phi_1}n_{A_2,\phi_2}
\label{cardl}
\end{equation}

By hypothesis, for $i=1,2$, there exists $B_i$, an abelian gorup of order $n_{A_i,\phi_i}$ such that $(A_i,\phi_i)\leq B_i\times B_i^*$. Whence :

\begin{align*}
(A,\phi)&=(A_1,\phi_1)\overset{\perp}{\oplus}(A_2,\phi_2)\\
&\leq (B_1\times B_1^*)\overset{\perp}{\oplus}(B_2\times B_2^*)\\
&\leq B\times B^*\text{ where $B:=B_1\times B_2$}
\end{align*}

We have included $(A,\phi)$ (as an alternate module) in $B\times B^*$, since $|B|=|B_1||B_2|=n_{A_1,\phi_1}n_{A_2,\phi_2}=n_{A,\phi}$ by equation \ref{cardl}, it follows, by definition, that $(A,\phi)$ is subsymplectic.

\end{proof}
 
As a result :

\begin{cor}\label{ptors}

Let $(A,\phi)$ be an alternate module then, we denote, for $p$ any prime dividing $|A|$, $S_p$ the unique $p$-Sylow of $A$. If for all $p$ dividing $|A|$, the induced submodule on $S_p$ by $\phi$ is subsymplectic then $(A,\phi)$ is subsymplectic.

\end{cor}
 
\begin{proof}

It is a well known result that if $p_1,\dots,p_r$ are the $r$ distinct primes dividing $|A|$ then, as a group, $A$ is isomorphic to the product of its $p$-Sylows

$$A=S_{p_1}\oplus\cdots \oplus S_{p_r} $$

Furthermore, if $1\leq i\neq j\leq r$ and $a_i\in S_{p_i}$, $a_j\in S_{p_j}$ then $\phi(a_i,a_j)$ is of order dividing some power of $p_i$ and some power of  $p_j$, it follows that  $\phi(a_i,a_j)$ is of order $1$, i.e. is trivial. In particular, for $i\neq j$, the induced modules $S_{p_i}$ and $S_{p_j}$ are orthogonal to each other. It follows that :

$$(A,\phi)=S_{p_1}\overset{\perp}{\oplus}\cdots \overset{\perp}{\oplus} S_{p_r} $$

The corollary is a direct consequence of this decomposition and proposition \ref{decomporth}.\end{proof}

With those classical properties, we will show the theorem.

\bigskip

\section{Any alternate module is subsymplectic}

\bigskip

Basically, the idea is to make an induction on the cardinal of the kernel. The following lemma is the major step of the proof  :

\begin{fl}

Let $p$ be a prime number and $(A,\phi)$  an alternate module  which is not symplectic with $A$ a $p$-group. Then there exists an alternate module $(\hat{A},\hat{\phi})$ such that $(A,\phi)$ is a submodule of $(\hat{A},\hat{\phi})$, $|A|<|\hat{A}|$ and $n_{A,\phi}=n_{\hat{A},\hat{\phi}}$.

\end{fl}

In the sequel, we will say that $(\hat{A},\hat{\phi})$ is an \textbf{extension of modules} of $(A,\phi)$ via the inclusion $\iota_A$ if $\iota_A:A\rightarrow \hat{A}$ is an inclusion of alternate modules, i.e. $\iota_{A*}\hat{\phi}=\phi$ and $\iota_A$ is an inclusion of modules. The extension will be \textbf{with constant Lagrangians} if $n_{\hat{A},\hat{\phi}}=n_{A,\phi}$.

 \bigskip

 The first subsection gathers some preliminary results. The second subection is the proof of the fundamental lemma. For subsections 3.1 and 3.2, $p$ is a fixed prime number and any alternate module $(A,\phi)$ is assumed to have its underlying group $A$ to be a $p$-group. In the third subsection, we prove the theorem.

\bigskip

\subsection{Preliminaries}

\bigskip

The next lemma is interesting in itself, it generalizes the second point of proposition \ref{3.sympB}.

\begin{lemme}\label{decompsymp}
Let $(A,\phi)$ be an alternate module. If there exists a submodule $B$ of $(A,\phi)$ such that $B$ is a symplectic module then :

$$(A,\phi)\text{ is isometric to } B\overset{\perp}{\oplus} B^{\perp} $$

\end{lemme}

\begin{proof}

Let $(A,\phi)$ be an alternate module and $B$ be a submodule of $A$ which is symplectic. Then $B^{\perp}$ is also a submodule of $(A,\phi)$. Since $B$ is symplectic, $B^{\perp}\cap B$ is trivial. It follows that $B\oplus B^{\perp}$ is a subgroup of $A$. By definition, $B$ is orthogonal to $B^{\perp}$ it follows that :

$$B\overset{\perp}{\oplus} B^{\perp} \text{ is a submodule of } (A,\phi)$$

Let $\pi$ be the canonical projection of $A$ on $A/K_{\phi}$. Since $B$ is symplectic, $B\cap K_{\phi}$ is trivial and $B$ is isomorphic to $\overline{B}:=\pi(B)$. Using the first point of proposition \ref{3.sympB} :

\begin{align*}
\overline{B}^{\perp}&=\frac{|A/K_{\phi}|}{|\overline{B}|}\\
&=\frac{|A|}{|K_{\phi}||B|}
\end{align*}

Clearly, $\pi^{-1}(\overline{B}^{\perp})=B^{\perp}$, it follows that $|B^{\perp}|=|K_{\phi}||\overline{B}^{\perp}|$. Finally, we have that :

\begin{align*}
|B\overset{\perp}{\oplus} B^{\perp}|&=|B||B^{\perp}|\\
&=|B||K_{\phi}||\overline{B}^{\perp}|\\
&=|B||K_{\phi}|\frac{|A|}{|K_{\phi}||B|}\\
&=|A|
\end{align*}

Since $B\overset{\perp}{\oplus} B^{\perp}$ is a submodule of $(A,\phi)$ and both modules have the same cardinal, we have that :

$$(A,\phi)\text{ is isometric to } B\overset{\perp}{\oplus} B^{\perp} $$

\end{proof}

The next lemma gives a sufficient condition for an  extension of modules to be with constant Lagrangians :

\begin{lemme}\label{extlagconst}

Let $p$ be a prime number, $(\hat{A},\hat{\phi})$ be an alternate module   and $A$ be a subgroup of $\hat{A}$ of index $p$. We denote $(A,\phi)$ the induced submodule on $A$ by $\hat{\phi}$. If $K_{\phi}$ is not included in $K_{\hat{\phi}}$ then $n_{A,\phi}=n_{\hat{A},\hat{\phi}}$.

\end{lemme}

\begin{proof}

Let $\hat{e}$ be an element of $\hat{A}$ which is not in $A$. Since $[\hat{A}:A]$ is of cardinal $p$, any element in $\hat{a}\in\hat{A}$ can uniquely be written as :

$$\hat{a}=\lambda\hat{e}+a\text{ where }0\leq \lambda\leq p-1\text{ and } a\in A $$

Define :

$$f:\left| \begin{array}{rcl}
K_{\phi}&\longrightarrow & \mathbb{Q}/\mathbb{Z}\\
k&\longmapsto&\hat{\phi}(\hat{e},k)\end{array}\right.$$

Then for all $k\in K_{\phi}$ :

\begin{align*}
p\cdot f(k)&=p\cdot \hat{\phi}(\hat{e},k)\\
&=\hat{\phi}(p\hat{e},k)\\
&=\phi(p\hat{e},k)\text{ since $p\hat{e}$ and $k$ are in $A$}\\
&=0\text{ since $k\in K_{\phi}$ and $p\hat{e}\in A$}
\end{align*}

As a result $Im(f)$ is a finite subgroup of $\mathbb{Q}/\mathbb{Z}$ of exponent dividing $p$. Whence $Im(f)$ is either trivial or the unique cyclic subgroup of order $p$ in $\mathbb{Q}/\mathbb{Z}$.

\bigskip

If $Im(f)$ is trivial, then $K_{\phi}$ is orthogonal to $\hat{e}$, since $K_{\phi}$ is orthogonal to $A$ and $\hat{A}$ is generated by $\hat{e}$ and $A$, it follows that $K_{\phi}$ is orthogonal to $\hat{A}$, whence $K_{\phi}\leq K_{\hat{\phi}}$. By assumption, this is a contradiction. It follows that $Im(f)$ is cyclic of cardinal $p$.

\bigskip

Let $x=\lambda \hat{e}+a\in K_{\hat{\phi}}$ where $0\leq \lambda \leq p-1$ and $a\in A$. Let $k_0\in K_{\phi}$ such that $f(k_0)$ is of order $p$. Since $x\in  K_{\hat{\phi}}$, we have $\hat{\phi}(x,k_0)=0$. On the other hand :

\begin{align*}
\hat{\phi}(x,k_0)&=\lambda\hat{\phi}(\hat{e},k_0)+\hat{\phi}(a,k_0)\\
 &=\lambda f(k_0)+\phi(a,k_0)\\
 &=\lambda f(k_0)+0\text{ since $k_0\in K_{\phi}$}\\
\end{align*}

It follows that $\lambda f(k_0)=0$. Since $f(k_0)$ is of order $p$, we have that $p$ divides $\lambda$ and since $0\leq \lambda \leq p-1$, we end up with $\lambda=0$ which implies that $x\in A$. As a result, we have shown that $K_{\hat{\phi}}\leq K_{\phi}$. 

\bigskip

Finally :

\begin{align*}
n_{\hat{A},\hat{\phi}}&=\sqrt{|\hat{A}||K_{\hat{\phi}}|}\\
&=\frac{\sqrt{p|A||K_{\phi}|}}{\sqrt{[K_{\phi}:K_{\hat{\phi}}]}}\\
&=\sqrt{\frac{p}{[K_{\phi}:K_{\hat{\phi}}]}}n_{A,\phi}
\end{align*}

Let $L$ be a Lagrangian of $A$, then $L$ is isotropic in $\hat{A}$ and by proposition \ref{3.Lagrsymp}, there exists a Lagrangian $\hat{L}$ of $\hat{A}$ containing $L$. Using Lagrange's theorem, we have that $|L|$ divides $|\hat{L}|$ and using proposition \ref{3.carlag}, $n_{A,\phi}$ divides $n_{\hat{A},\hat{\phi}}$. It follows that $\sqrt{\frac{p}{[K_{\phi}:K_{\hat{\phi}}]}}$ is a positive integer which cannot but be equal to one (because $p$ is a prime number). In particular, we have that $n_{\hat{A},\hat{\phi}}=n_{A,\phi}$. \end{proof}

The last lemma in those preleminaries allows us to construct extension of modules in a fairly simple way.

\begin{lemme}\label{exten}

Let $p$ be a prime number, $(A,\phi)$  an alternate module where $A$ is a $p$ group. Assume that there exist $r+1$ non-trivial elements   $e,e_1,\dots,e_r$ in $A$ such that :

$$A=\langle e\rangle\times \langle e_1\rangle\times\cdots\times\langle e_r\rangle $$

Assume furthermore that for any $1\leq i\leq r$ the order of $\phi(e,e_i)$ is strictly lesser than the order of $e_i$ then there exists an alternate module $(\hat{A},\hat{\phi})$ such that $\hat{A}=\langle \hat{e}\rangle\times \langle e_1\rangle\times\cdots\times\langle e_r\rangle$ where $\hat{e}$ is of order $p$ times the order of $e$ and the inclusion :

$$\iota_A:\left| \begin{array}{rcl}
A&\longrightarrow &\hat{A}\\
e&\longmapsto&p\hat{e}\\
e_i&\longmapsto& e_i\text{ for $1\leq i\leq r$}\end{array}\right.$$

is an inclusion of submodules.

\end{lemme}

\begin{proof}

Let $\hat{A}$ be $\langle \hat{e}\rangle\times \langle e_1\rangle\times\cdots\times\langle e_r\rangle$ where $\hat{e}$ is of order $p$ times the order of $e$. Clearly, the mapping :

$$\iota_A:\left| \begin{array}{rcl}
A&\longrightarrow &\hat{A}\\
e&\longmapsto&p\hat{e}\\
e_i&\longmapsto& e_i\text{ for $1\leq i\leq r$}\end{array}\right.$$

defines an injective morphism of groups. In order to define the bilinear form $\hat{\phi}$,  it suffices to define it on a generating set of $\hat{A}\times \hat{A}$. Define :

$$\hat{\phi}(e_i,e_j):=\phi(e_i,e_j)\text{ for all $1\leq i,j\leq r$} $$

We also define $\hat{\phi}(\hat{e},\hat{e}):=0$. Finally, remark that the group $\mathbb{Q}/\mathbb{Z}$ is $p$-divisible (i.e. any element admits a $p$-th root). Then for any $1\leq i\leq r$, define $\lambda_i$ to be one element verifying $p\lambda_i=\phi(e,e_i)$. 

\bigskip

Remark that the order of $\lambda_i$ thus defined divides $p$ times the order of $\phi(e,e_i)$. By assumption, this divides the order of $e_i$ and also divides the order of $\hat{e}$ which is $p$ times the order of $e$. It follows that $\lambda_i$ divides both the order of $ \hat{e}$ and $e_i$.  Whence the following equations :

\begin{align*}
\hat{\phi}(e_i,e_j)&:=\phi(e_i,e_j)\text{ for all $1\leq i,j\leq r$}\\
\hat{\phi}(\hat{e},\hat{e})&:=0\\
\hat{\phi}(\hat{e},e_i)&:=\lambda_i\\
\hat{\phi}(e_i,\hat{e})&:=-\lambda_i\\
\end{align*}

allow us to define a group morphism on $\hat{A}$ which is clearly an alternate module. Let $\psi$ be the induced module on $A$ via $\iota_A$ by $\hat{\phi}$ :

\begin{align*}
\psi(e_i, e_j)&=\hat{\phi}(e_i,e_j)=\phi(e_i,e_j)\text{for all $1\leq i,j\leq r$}\\
\psi(e,e_j)&=\hat{\phi}(p\hat{e},e_j)=p\lambda_j=\phi(e,e_j)\text{ for all $1\leq j\leq r$}\\
\end{align*}

Whence $\psi=\phi$ on $A\times A$ (since they are alternate forms). In particular we have shown that $(A,\phi)$ is a submodule of $(\hat{A},\hat{\phi})$ via $\iota_A$. \end{proof}

We have  the tools to prove the fundamental lemma.

\bigskip

\subsection{Proof of the fundamental lemma}

\bigskip

We recall first the statement of the fundamental lemma :

\begin{fl}

Let $p$ be a prime number and $(A,\phi)$  an alternate module  which is not symplectic with $A$ a $p$-group. Then there exists an alternate module $(\hat{A},\hat{\phi})$ such that $(A,\phi)$ is a submodule of $(\hat{A},\hat{\phi})$, $|A|<|\hat{A}|$ and $n_{A,\phi}=n_{\hat{A},\hat{\phi}}$.

\end{fl}

\begin{proof}

We prove the lemma doing a strong induction on the cardinal of the module. Let $p^k$ be such that the lemma is true for any alternate module $M$ with $|M|<p^k$. Let $(A,\phi)$ be an alternate module which is not symplectic and $A$ be of cardinal $p^k$. There are two different cases to consider :

\bigskip

$$\text{ Case 1 : }K_{\phi}\text{ is not included in } p\cdot A$$

By the classification of abelian $p$-groups, decompose the group $A$ as a product of cyclic subgroups :

\begin{equation}
A=\langle e_0\rangle \times\langle e_1\rangle\times \cdots\times \langle e_r\rangle
\label{decompA}
\end{equation}

Denote $d_i$ the order of $e_i$, then we may assume that $d_{i-1}$ divides $d_i$ for $1\leq i\leq r$.

\bigskip

Let $k_0\in K_{\phi}$ which is not in $p\cdot A$. By the decomposition \ref{decompA}, we may write 

$$k_0=\sum_{i=0}^r\alpha_ie_i\text{ where }  \alpha_0,\dots,\alpha_r\in \mathbb{N} $$

Since $k_0$ is not in $p\cdot A$ there exists at least one coefficient among the $\alpha_0,\dots,\alpha_r$ which is not divisible by $p$. Denote $i_0$ to be one index such that $\alpha_{i_0}$ is not divisible by $p$. 

\bigskip

Define the group $\hat{A}:=\mathbb{Z}/p\times A$. Let $\gamma$ be a generator of $\mathbb{Z}/p$. Then we define an alternate form $\hat{\phi}$ on $\hat{A}$ by :

\begin{align*}
\hat{\phi}(e_i,e_j)&:=\phi(e_i,e_j)\text{ for $0\leq i,j\leq r$}\\
\hat{\phi}(\gamma,e_i)&:=0\text{ for $0\leq i\leq r$ and $i\neq i_0$}\\
\hat{\phi}( e_i,\gamma)&:=0\text{ for $0\leq i\leq r$ and $i\neq i_0$}\\
\hat{\phi}(\gamma,\gamma)&:=0\\
\hat{\phi}(\gamma,e_{i_0})&:=\frac{1}{p} \\
\hat{\phi}( e_{i_0},\gamma)&:=-\frac{1}{p} \\
\end{align*}

Clearly, $\hat{\phi}$ is an alternate form on $\hat{A}$, furthermore, the inclusion of $A$ in $\hat{A}$ given by $a$ is mapped to $(0,a)$ clearly leads to an inclusion of modules.  Finally :

\begin{align*}
\hat{\phi}(\gamma,k_0)&=\sum_{i=0}^r\alpha_i\hat{\phi}(\gamma,e_i)\\
&=\alpha_{i_0}\hat{\phi}(\gamma,e_{i_0})\text{ since $\gamma$ is orthogonal to $e_i$ where $i\neq i_0$}\\
&=\frac{\alpha_{i_0}}{p}\neq 0\text{ in $\mathbb{Q}/\mathbb{Z}$ since $p$ does not divide $\alpha_{i_0}$}
\end{align*}

Hence we found an element in $K_{\phi}$ which is not in the kernel of $K_{\hat{\phi}}$. Applying lemma \ref{extlagconst}, we have $n_{A,\phi}=n_{\hat{A},\hat{\phi}}$, hence $(\hat{A},\hat{\phi})$ contains $(A,\phi)$, has the same cardinality of Lagrangians and is strictly greater than $A$. In this case, we have constructed the extension with constant Lagrangians (without using the induction hypothesis).

$$\text{ Case 2 : }K_{\phi}\text{ is included in } p\cdot A$$

Write :

\begin{equation}
A=\langle e \rangle \times\langle e_1\rangle\times \cdots\times \langle e_r\rangle
\label{decompA2}
\end{equation}

Denote $d$ the order of $e$, $d_i$ the order of $e_i$, then we may assume that $d$ divides $d_1$ and $d_{i}$ divides $d_{i+1}$ for $1\leq i\leq r-1$.

\bigskip

Assume that there exists $1\leq i\leq r$ such that the order of $\phi(e,e_i)$ is equal to $d_i$ (the order of $e_i$), since the order of $\phi(e,e_i)$ divides $d$, the order of $e$ (by bilinearity), it follows that $d_i$ divides $d$. Since we also have $d$ divides $d_i$ by definition, we have two elements $e$ and $e_i$ in $A$ whose order is equal to $d$ and such that $\phi(e,e_i)$ is of order $d$. As a result, the submodule 
$B$ generated by $e$ and $e_i$ is isometric to $\mathbb{Z}/d\times (\mathbb{Z}/d)^*$. In particular $B$ is a submodule of $(A,\phi)$ which is symplectic, by lemma \ref{decompsymp}, it follows that :

$$(A,\phi)=B\overset{\perp}{\oplus}B^{\perp} $$

Let $(B^{\perp},\psi)$ be the induced submodule on $B^{\perp}$ by $\phi$. By induction hypothesis, there exists a module $(C,\hat{\psi})$ such that $n_{B^{\perp},\psi}=n_{C,\hat{\psi}}$ with $|B^{\perp}|<|C|$ and $(B^{\perp},\psi)$ is a submodule of $(C,\hat{\psi})$.  Denote :

$$(\hat{A},\hat{\phi}):=B\overset{\perp}{\oplus} (C,\hat{\psi})$$

Clearly $A=B\overset{\perp}{\oplus}B^{\perp}$ is a submodule of $(\hat{A},\hat{\phi})$ and $|A|<|\hat{A}|$. Remark that $K_{\phi}=K_{\psi}$ and $K_{\hat{\phi}}=K_{\hat{\psi}}$ since $B$ is symplectic. Whence :

\begin{align*}
n_{A,\phi}&=\sqrt{|A||K_{\phi}|}\\
&=\sqrt{|B|}\sqrt{|B^{\perp}||K_{\psi}|}\\
&=\sqrt{|B|}n_{B^{\perp},\psi}\\
&=\sqrt{|B|}n_{C,\hat{\psi}}\text{ by construction of $(C,\hat{\psi})$}\\
&=\sqrt{|B|}\sqrt{|C||K_{\hat{\psi}}|}\\
&= \sqrt{|\hat{A}||K_{\hat{\phi}}|}\\
&=n_{\hat{A},\hat{\phi}}
\end{align*}

In this case we can also construct an extension $(\hat{A},\hat{\phi})$ of $(A,\phi)$ with constant Lagrangians.

\bigskip

Assume now that for $1\leq i\leq r$, the order of $\phi(e,e_i)$ strictly divides the order of $d_i$. Then we are in the condition of applications of lemma \ref{exten}. Define $\hat{A}=\langle \hat{e}\rangle\times \langle e_1\rangle\times\cdots\times\langle e_r\rangle$ where $\hat{e}$ is of order $p$ times the order of $e$ and the inclusion :

$$\iota_A:\left| \begin{array}{rcl}
A&\longrightarrow &\hat{A}\\
e&\longmapsto&p\hat{e}\\
e_i&\longmapsto& e_i\text{ for $1\leq i\leq r$}\end{array}\right.$$

then we have defined on $\hat{A}$ a bilinear form $\hat{\phi}$ such that $\iota_{A*}\hat{\phi}=\phi$. 

\bigskip

We recall that we are in the second case where $K_{\phi}\leq p\cdot A$. Denote $\pi$ the canonical projection of $A$ onto $A/K_{\phi}$.  Since $K_{\phi}\leq p\cdot A$, it follows that $\pi^{-1}(p\cdot (A/K_{\phi}))=p\cdot A$. In particular (since $e$ is clearly not an element of $p\cdot A$) it follows that $\pi(e)$ is not in $p\cdot (A/K_{\phi})$. 

\bigskip

We know that $A/K_{\phi}$ is isomorphic to $B\times B^*$ by corollary \ref{3.classificationsymp}. Denote $g_1,\dots,g_s$ a "base" of $B$ and $g_1^*,\dots,g_s^*$ the corresponding "dual base" of $B^*$ and write :

$$\pi(e)=\sum_{i=1}^r\gamma_ig_i+\gamma_i^*g_i^* $$

Since $\pi(e)$ is not in $p\cdot (A/K_{\phi})$, it follows that at least one of the $\gamma_i$ or $\gamma_i^*$ is not divisible by $p$. Up to exchanging $g_i$ and $g_i^*$, we may assume that $\gamma_{i_0}^*$ is not divisible by $p$.

\bigskip

Let $p^k$ (with $k\geq 1$) be the order of $g_{i_0}$ and $a_0\in A$ such that $\pi(a_0)=g_{i_0}$. Since $g_{i_0}$ is of order $p^k$, $p^k\cdot a_0\in K_{\phi}$. Furthermore :

\begin{align*}
\phi(e,p^{k-1}\cdot a_0)&=\overline{\phi}(\pi(e),p^{k-1}\cdot\pi(a_0))\text{ by definition of $(A/K_{\phi},\overline{\phi})$}\\
&=p^{k-1}\overline{\phi}(\pi(e),g_{i_0})\\
&=p^{k-1}\sum_{i=1}^r\gamma_i\overline{\phi}(g_i,g_{i_0})+\gamma_i^*\overline{\phi}(g_i^*,g_{i_0})\\
&=p^{k-1}\gamma_{i_0}^*\overline{\phi}(g_{i_0}^*,g_{i_0})\\
\end{align*}

Remark that, by definition $\overline{\phi}(g_{i_0}^*,g_{i_0})$ is of order $p^k$, since $\gamma_{i_0}^*$ is not divisible by $p$ :

$$\phi(e,p^{k-1}\cdot a_0)=p^{k-1}\gamma_{i_0}^*\overline{\phi}(g_{i_0}^*,g_{i_0})\text{ is of order $p$, whence not trivial.}$$

Remark now that, in $(\hat{A},\hat{\phi})$ we have :

\begin{align*}
\hat{\phi}(\hat{e},p^k\cdot a_0)&=\hat{\phi}(p\cdot\hat{e},p^{k-1}\cdot a_0)\text{ by $\mathbb{Z}$-bilinearity}\\
&=\phi(e,p^{k-1}\cdot a_0)\text{ since $p\cdot \hat{e}=\iota_A(e)$ and $\iota_{A*}\hat{\phi}=\phi$}\\
&\neq 0\text{ in $\mathbb{Q}/\mathbb{Z}$}
\end{align*}

In particular, $p^k\cdot a_0$ is an element of $K_{\phi}$ which is not in $K_{\hat{\phi}}$. Applying lemma \ref{extlagconst}, we have that $n_{A,\phi}=n_{\hat{A},\hat{\phi}}$.  \end{proof}

Now that the fundamental lemma is proven, we shall see that the theorem follows by an easy induction on the cardinal of the kernel of the alternate module.

\bigskip

\subsection{Proof of the theorem}

\bigskip

Recalling that the result we want to prove is the following theorem :

\begin{thm} 
Any alternate module is subsymplectic.

\end{thm}

\begin{proof} 

By corollary \ref{ptors}, it suffices to prove it for alternate modules $(A,\phi)$ where $A$ is a $p$-group and $p$ is a prime number.

\bigskip

Let us show the following by induction on $k$ : Let $(A,\phi)$ be an alternate module which is a $p$-group with $|K_{\phi}|=p^k$ then $(A,\phi)$ is subsymplectic. 

\bigskip

If $k=0$ then $(A,\phi)$ is symplectic. Therefore, by corollary \ref{3.classificationsymp}, it is isometric  to $B\times B^*$ where $|B|=\sqrt{|A|}=n_{A,\phi}$. It follows that $(A,\phi)$ is subsymplectic by definition.

\bigskip

If $k>0$ then $(A,\phi)$ is an alternate module which is not symplectic (its kernel is not trivial) whence, by the fundamental lemma, there exists an alternate module $(\hat{A},\hat{\phi})$ such that $(A,\phi)$ is a submodule of $(\hat{A},\hat{\phi})$, $|A|<|\hat{A}|$ and $n_{A,\phi}=n_{\hat{A},\hat{\phi}}$.

\bigskip

We have $n_{A,\phi}^2=n_{\hat{A},\hat{\phi}}^2$ so $|A||K_{\phi}|=|\hat{A}||K_{\hat{\phi}}|$, since $|A|<|\hat{A}|$ we have that $|K_{\hat{\phi}}|<|K_{\phi}|=p^k$. It follows that we can apply the induction hypothesis to $(\hat{A},\hat{\phi})$ and there exists an abelian group $B$ of order $n:=n_{\hat{A},\hat{\phi}}$ such that :

$$(\hat{A},\hat{\phi})\leq B\times B^* $$

Since $(A,\phi)$ is a submodule of $(\hat{A},\hat{\phi})$, we finally have included the module $(A,\phi)$ in $B\times B^*$ where $|B|=n=n_{\hat{A},\hat{\phi}}=n_{A,\phi}$. It follows that $(A,\phi)$ is subsymplectic by definition. \end{proof}

\bigskip

\section{Remarks}

\bigskip

The proof is constructive. Indeed, it explicitely gives the construction (by induction) of a symplectic module $B\times B^*$ containing $(A,\phi)$ with $|B|=n_{A,\phi}$. Basically, in order to define an algorithm computing $B$, one needs an algorithm that, given $(A,\phi)$ computes its kernel and a "diagonalization" of its associated symplectic module $(A/K_{\phi},\overline{\phi})$.

\bigskip

Remark that, given $(A,\phi)$, an alternate module, there might exist $B_1\neq B_2$, both of cardinal $n_{A,\phi}$ such that $(A,\phi)$ is included in $B_i\times B_i^*$ for $i=1,2$ :

\bigskip

\begin{cex}

Let $A$ be the abelian group $\mathbb{Z}/2\times\mathbb{Z}/4$ endowed with the trivial bilinear form. It follows that $n_{A,\phi}=|A|$. Let $B_1:=A$ and $B_2:=\mathbb{Z}/8$. Then we have two inclusions :

\begin{displaymath}
\iota_{A,1}:
\left|
  \begin{array}{rcl}
    A& \longrightarrow &B_1\times B_1^{*}\\
 (u,v)& \longmapsto &((u,v),(0,0))\\
  \end{array}
\right.\text{ and }\iota_{A,2}:
\left|
  \begin{array}{rcl}
    A& \longrightarrow &B_2\times B_2^{*}\\
 (u,v)& \longmapsto &(4u,2v)\\
  \end{array}
\right.
\end{displaymath}

Those two inclusions are inclusions of submodules (when $B_i\times B_i^*$ is endowed with its natural structure of symplectic form, defined in example \ref{3.sympmod}).  Furthermore, it is clear that they are extensions with constant Lagrangians. However $B_1\neq B_2$.

\end{cex}

As we have stated in the introduction of this paper, the theorem will be used in another paper in order to give a classification of centralizer of irreducible subgroups of $PSL(n,\mathbb{C})$. An interesting thing to have would be a classification of alternate modules. 

\bigskip

Let $(A,\phi)$ be an alternate module, we say that $(A,\phi)$ is \textbf{indecomposable} if, whenever $(A,\phi)$ is the orthogonal sum of $(A_1,\phi_1)$ and $(A_2,\phi_2)$ then $(A_1,\phi_1)$ or $(A_2,\phi_2)$ is $(\{0\},0)$. 

\bigskip

Since any alternate module is clearly the sum of indecomposable ones by induction, if we want to classify alternate modules, we only need to characterize indecomposable ones. This leads to the following question :

\begin{q}

Let $(A,\phi)$ be an indecomposable alternate module. Does it follow that $rk(A)\leq 3$? Does it follow that $rk(A/K_{\phi})\leq 2$?
\end{q}

We computed many examples and this conjecture seems to be verified. However we still do not have a proof for it.

\newpage

\section*{Acknowledgements}

\bigskip

All computations performed on examples leading to this result have been performed using the GAP system  \cite{GAP4}.

\bigskip

I would like to thank my thesis advisor Olivier Guichard, whose suggestions have been very helpful. I would also like to thank Gaël Collinet who encouraged me to fully investigate this topic and Jean-Pierre Tignol for pointing out some helpful references. Finally, I would like to thank my fellow Ph.D student, Mohamad Maassarani for carefully listening to my proofs and his helpful comments.

\end{document}